\documentclass[a4j,11pt,leqno]{amsart}
\usepackage{amsmath,amssymb,amscd, amsthm}
\usepackage{epic,eepic,epsfig}
\usepackage{mathrsfs}
\usepackage[all,cmtip]{xy}
\usepackage{graphicx}
\usepackage{here}
 \usepackage[nodayofweek]{datetime}

\makeatletter
\def\@settitle{\begin{center}
  \baselineskip14\p@\relax
  \normalfont\LARGE\bfseries
  \@title
  \ifx\@subtitle\@empty\else
     \\[1ex] 
     
     \normalsize\mdseries\@subtitle
  \fi
 \ifx\@didication\@empty\else
     \\[2ex] 
     
     \large\mdseries\it\@dedication
  \fi
  \end{center}
}
\def\subtitle#1{\gdef\@subtitle{#1}}
\def\@subtitle{}
\def\dedication#1{\gdef\@dedication{#1}}
\def\@dedication{}
\makeatother

\usepackage{wrapfig}

\makeatletter

\makeatletter
\renewcommand{\section}{\@startsection
{section}{1}{0mm}{5mm}{2mm}{\raggedright\bfseries}}

\makeatother

\newtheorem{theorem}{Theorem}[section] 
\newtheorem{Theorem}[theorem]{Theorem}
\newtheorem{Lemma}[theorem]{Lemma}
\newtheorem{Corollary}[theorem]{Corollary}

\theoremstyle{definition}
\newtheorem{Definition}[theorem]{Definition}
\newtheorem{Example}[theorem]{Example}

\newtheorem{Remark}[theorem]{Remark}

\newtheorem{Notation}[theorem]{Notation}

\newtheorem{Note}[theorem]{Note}

\begin{document}
\DeclareRobustCommand{\magnus}{\genfrac\{\}{0pt}{}}
\def\bk{{\mathbf{k}}}
\def\MM{{\mathsf{M}}}
\def\SS{{\mathsf{S}}}
\def\dep{\mathrm{dep}}
\def\bt{{\mathbf{t}}}
\def\bs{{\mathbf{s}}}
\def\bu{{\mathbf{u}}}
\def\bbr{{\mathbf{r}}}
\def\bkappa{{\boldsymbol \kappa}}
\def\btau{{\boldsymbol \tau}}
\def\brho{{\boldsymbol \rho}}
\def\lp{\langle}
\def\rp{\rangle}
\def\blambda{{\boldsymbol \lambda}}
\def\bmu{{\boldsymbol \mu}}
\def\N{{\mathbb N}}
\def\C{{\mathbb C}}
\def\Z{{\mathbb Z}}
\def\R{{\mathbb R}}
\def\Q{{\mathbb Q}}
\def\bQ{{\overline{\mathbb Q}}}
\def\Gal{{\mathrm{Gal}}}
\def\et{\text{\'et}}
\def\ab{\mathrm{ab}}
\def\proP{{\text{pro-}p}}
\def\padic{{p\mathchar`-\mathrm{adic}}}
\def\la{\langle}
\def\ra{\rangle}
\def\scM{\mathscr{M}}
\def\lala{\la\!\la}
\def\rara{\ra\!\ra}
\def\ttx{{\mathtt{x}}}
\def\tty{{\mathtt{y}}}
\def\ttz{{\mathtt{z}}}
\def\bkappa{{\boldsymbol \kappa}}
\def\scLi{{\mathscr{L}i}}
\def\sLL{{\mathsf{L}}}
\def\cY{{\mathcal{Y}}}
\def\ad{\mathrm{ad\,}}
\def\Spec{\mathrm{Spec}\,}
\def\Ker{\mathrm{Ker}}
\def\CHplus{\underset{\mathsf{CH}}{\oplus}}
\def\check{{\clubsuit}}
\def\kaitobox#1#2#3{\fbox{\rule[#1]{0pt}{#2}\hspace{#3}}\ }
\def\vru{\,\vrule\,}
\newcommand*{\longhookrightarrow}{\ensuremath{\lhook\joinrel\relbar\joinrel\rightarrow}}
\newcommand{\hooklongrightarrow}{\lhook\joinrel\longrightarrow}
\def\nyoroto{{\rightsquigarrow}}
\newcommand{\pathto}[3]{#1\overset{#2}{\dashto} #3}
\newcommand{\pathtoD}[3]{#1\overset{#2}{-\dashto} #3}
\def\dashto{{\,\!\dasharrow\!\,}}
\def\ovec#1{\overrightarrow{#1}}
\def\isom{\,{\overset \sim \to  }\,}
\def\GT{{\widehat{GT}}}
\def\bfeta{{\boldsymbol \eta}}
\def\brho{{\boldsymbol \rho}}
\def\sha{\scalebox{0.6}[0.8]{\rotatebox[origin=c]{-90}{$\exists$}}}
\def\upin{\scalebox{1.0}[1.0]{\rotatebox[origin=c]{90}{$\in$}}}
\def\downin{\scalebox{1.0}[1.0]{\rotatebox[origin=c]{-90}{$\in$}}}
\def\torusA{{\epsfxsize=0.7truecm\epsfbox{torus1.eps}}}
\def\torusB{{\epsfxsize=0.5truecm\epsfbox{torus2.eps}}}

\title{
Demi-shuffle duals of Magnus polynomials \\
in a free associative algebra}

\author{Hiroaki Nakamura (Osaka Univeristy)}

\subjclass{16S10; 05A10, 11G55, 68R15}
\keywords{shuffle product, non-commutative polynomial, group-like series}

\address{Department of Mathematics,
Graduate School of Science,
Osaka University,
Toyonaka, Osaka 560-0043, Japan}
\email{nakamura@math.sci.osaka-u.ac.jp}

\maketitle

\markboth{H.Nakamura}
{Demi-shuffle duals of Magnus polynomials}

\begin{abstract}
We study two linear bases of the free associative algebra
$\Z\la X,Y\ra$: one is formed by the Magnus polynomials of type
$(\mathrm{ad}_X^{k_1}Y)\cdots(\mathrm{ad}_X^{k_d}Y) X^k$ 
and the other is its dual basis (formed by what we call the `demi-shuffle' polynomials)
with respect to the standard pairing on the monomials of $\Z\la X,Y\ra$.
As an application, we show a formula of Le-Murakami, Furusho type 
that expresses arbitrary coefficients of a group-like series $J\in \C\lala X,Y\rara$
in terms of the ``regular'' coefficients of $J$.  
\end{abstract}

\tableofcontents

\section{Introduction}

Let $R$ be a commutative integral domain of characteristic 0, and let $R\la X,Y\ra$ be
the free associative algebra generated over $R$ by two (non-commutative) 
letters $X$ and $Y$.
For $u,v\in R\la X,Y\ra$, we shall write $[u,v]$
to denote the Lie bracket $uv-vu$.
In \cite{M37}, W.Magnus introduced the associative subalgebra 
$S_X\subset R\la X,Y\ra$ generated by 
(what are called) the elements arising by elimination of $X$:
\begin{equation}
Y^{(0)}:=Y, \quad
Y^{(k+1)}:=[X,Y^{(k)}] \quad
(k=0,1,2,\dots),
\end{equation}
and showed that $S_X$ is freely generated by the $Y^{(k)}$
$(k=0,1,2,\dots)$.
Moreover, he derived that every element $Z$ of  $R\la X,Y\ra$ can be
written uniquely in the form
\begin{equation} \label{1st-elimination}
Z=\alpha_0 X^m +s_1 X^{m-1}+\cdots+s_m,
\end{equation}
where $\alpha_0\in R$, $s_1,\dots, s_m\in S_X$
(see \cite[Hilfssatz 2]{M37}, \cite[Lemma 5.6]{MKS}).
This observation is the first step preceding to repeated elimination
for the construction 
of the basic Lie elements (an ordered basis of free Lie algebra)
whose powered products in decreasing orders give 
Poincar\'e-Birkoff-Witt
basis of the enveloping algebra $R\la X,Y\ra$ (\cite[Theorem 5.8]{MKS}).
Apparently, this theory was historically a starting point 
toward subsequent developments of
finer constructions of free Lie algebra bases 
due to Lazard, Hall, Lyndon, Viennot and others
(cf. e.g., \cite[Notes 4.5, 5.7]{Reu}). 

In this note, we however stay on the first step of elimination 
(\ref{1st-elimination}) and 
look at 
combinatorial properties of a certain basis 
$\{\MM^{(\bk)}\}_{\bk\in \N_0^{(\infty)}}$ of $R\la X,Y\ra$
(to be called the Magnus polynomials below)
designed as follows: 

\begin{Notation} 
Let $\N_0$ denote the set of non-negative integers, and 
let 
$$
\N_0^{(\infty)}:=\bigcup_{d=0}^\infty
\left(
\prod_{k=1}^{d}\N_0\right)
\times \N_0 
$$ 
be the collection of finite sequences 
$
\bk=(k_1,\dots,k_d; k_\infty)
$
of non-negative integers equipped 
with a special last entry $k_\infty\in \N_0$. 
Here, we consider $(;k_\infty)$ also as elements of 
$\N_0^{(\infty)}$ coming from $d=0$.
For $\bk\in \N_0^{(\infty)}$, define
$|\bk|:=\sum_{i=1}^\infty k_i=k_1+\cdots+k_d+k_\infty$ 
(resp. $\dep(\bk):=d$),
and call it the size (resp. depth) of $\bk$.
\end{Notation}

\begin{Definition}[Magnus polynomial]
\label{defMagnus}
For $\bk=(k_1,\dots,k_d; k_\infty)\in\N_0^{(\infty)}$, define
$$
\MM^{(\bk)}:=Y^{(k_1)}\cdots Y^{(k_d)} \cdot X^{k_\infty} \in R\la X,Y\ra.
$$
We also set $\MM^{(;0)}=1$, $\MM^{(;k)}=X^k$ ($k=1,2,\dots)$.
Note that $\MM^{(k;0)}=Y^{(k)}$ for $k\ge 0$.
\end{Definition}

\begin{Example} 
$\MM^{(1,0;2)}=Y^{(1)}Y^{(0)}X^2=(XY-YX)YX^2=XY^2X^2-YXYX^2$.
\end{Example}

It is not difficult to see that the Magnus polynomial $\MM^{(\bk)}\in R\la X,Y\ra$ 
is homogeneous
of bidegree $(|\bk|,\dep(\bk))$ in $X$ and $Y$.

The above mentioned Magnus expression (\ref{1st-elimination})
can then be rephrased as
\begin{equation} \label{Magnus}
Z=\sum_{\bk\in \N_0^{(\infty)}}
\alpha_{\bk}\, \MM^{(\bk)}
\end{equation}
with uniquely determined coefficients $\alpha_{\bk}\in R$ for any given 
$Z\in R\la X,Y\ra$.
In other words, the collection $\{\MM^{(\bk)}\mid \bk\in  \N_0^{(\infty)}\}$ 
forms an $R$-linear basis of $R\la X,Y\ra$.

Below in \S 2, we will construct another $R$-linear basis 
$\{\SS^{(\bk)}\mid \bk\in  \N_0^{(\infty)}\}$
(formed by what we call the `demi-shuffle' polynomials)
and show that 
$\{\MM^{(\bk)}\}_\bk$
and $\{\SS^{(\bk)}\}_\bk$
are dual to each other under the standard pairing
with respect to the monomials of $R\la X,Y\ra$
(Theorem \ref{orthogonality}).
We then in \S 3 shortly generalize the duality to the case of
free associative algebras of more variables (Theorem 
\ref{orthogonality2}). 
In \S 4, we apply the formation of dual basis 
to derive a formula of Le-Murakami, Furusho type 
that expresses arbitrary coefficients of a group-like series 
$J\in R\lala X,Y\rara$ in terms of 
the ``regular'' coefficients of $J$
(Theorem \ref{LeMuFu}).  

\section{Demi-shuffle duals and array binomial coefficients}

Let $W$ be the subset of $R\la X,Y\ra$ formed by the monomials in $X,Y$ 
together with $1$, and call any element of $W$ a word. 
It is clear that $W$ forms a free monoid by the concatenation product
that restricts the multiplication of $R\la X,Y\ra$.
Each element of $R\la X,Y\ra$ is an $R$-linear combination of words in $W$.
For two elements $u,v\in R\la X,Y\ra$, define the standard pairing $\lp u,v\rp\in R$
so as to extend $R$-linearly the Kronecker symbol $\lp w,w'\rp:=\delta_{w}^{w'}\in\{0,1\}$
for words $w,w'\in W$.

\begin{Notation} We use the notation $w_\bk:=X^{k_1}Y\cdots X^{k_d}YX^{k_\infty}$
and call it the word associated to $\bk=(k_1,\dots,k_d; k_\infty)\in\N_0^{(\infty)}$.
The mapping $\bk\mapsto w_\bk$ gives a bijection between $\N_0^{(\infty)}$ onto $W$.
(Note that $w_{(;0)}=1$.)
The standard pairing $\lp w_\bk,w_{\bk'}\rp$ is equal to $0$ or $1$ 
according to whether $\bk\ne \bk'$ or $\bk=\bk'$.
\end{Notation}

The purpose of this section is to describe the dual of the 
Magnus basis
$\{\MM^{(\bk)}\}_{\bk\in \N_0^{(\infty)}}$ with respect to the standard
pairing.

\begin{Definition}[Demi-shuffle polynomial]
\label{demi-shuffle}
For $\bk=(k_1,\dots,k_d; k_\infty)\in\N_0^{(\infty)}$, define
$$
\SS^{(\bk)}:=
(\cdots((X^{k_1}Y)\sha X^{k_2})Y)\sha\cdots)\sha  X^{k_d})Y)\sha
X^{k_\infty} \in R\la X,Y\ra,
$$
where $\sha$ denotes the usual shuffle product. 
We also set $\SS^{(;0)}=1$, $\SS^{(;k)}=X^k$ ($k=1,2,\dots)$.
Note that $\SS^{(k;0)}=X^kY$ for $k\ge 0$.
\end{Definition}

The construction of $\SS^{(\bk)}$ can be interpreted 
as forming the linear sum
of all words obtained from 
the word
$w_\bk=X^{k_1}Y\cdots X^{k_d}YX^{k_\infty}$
by consecutively applying `left shuffles' of letters $X$'s 
and `concatenations' of letters $Y$'s in $w_\bk$.

\begin{Example}
Here are a few examples:
$\SS^{(0,1;0)}=(Y\sha X)Y=YXY+XYY$;
$\SS^{(1,1;0)}=((XY)\sha X)Y=XYXY+2XXYY$;
$\SS^{(1,0,1;0)}=(((XY)Y)\sha X)Y=XYYXY+XYXY^2+2X^2Y^3$.
Using the first identity, one can also compute
\begin{align*}
&\SS^{(0,1;1)}=((Y\sha X)Y)\sha X=(YXY+XYY)\sha X \\
&=(YXYX+2YXXY+XYXY)+(XYYX+XYXY+2XXYY) \\
&=2XXYY+2XYXY+XYYX+2YXXY+YXYX.
\end{align*}
\end{Example}

\begin{Theorem}[Duality] \label{orthogonality}
For $\bt, \bk\in \N_0^{(\infty)}$, we have
$$
\lp \SS^{(\bt)},\MM^{(\bk)}\rp=\delta_\bt^{\bk}.
$$
Here $\delta_\bt^{\bk}$ is the Kronecker symbol, i.e., designating $0$ or $1$
according to whether $\bt\ne\bk$ or $\bt=\bk$ respectively.
\end{Theorem}

Before going to the proof of the above theorem, we introduce the following
notation.

\begin{Definition}[Array binomial coefficient]
For $\bt,\bk\in  \N_0^{(\infty)}$ with $\dep(\bt)=\dep(\bk)$, 
$|\bt|=|\bk|$, define
\begin{equation} \label{ArrayBino}
\binom{\bt}{\bk}:=
\binom{t_1}{k_1}
\binom{t_1+t_2-k_1}{k_2}
\cdots
\binom{t_1+\cdots t_d-k_1-\cdots-k_{d-1}}{k_d},
\end{equation}
where $\bt=(t_1,\dots,t_d,t_\infty)$,
$\bk=(k_1,\dots,k_d,k_\infty)$.
We understand $\binom{\bt}{\bk}=1$ if
$\bt=\bk=(;N)$ for some $N\in \N_0$.
We set $\binom{\bt}{\bk}:=0$ if either $\dep(\bt)\ne\dep(\bk)$ or 
$|\bt|\ne|\bk|$ holds.
\end{Definition}

\begin{Remark} 
\label{rem2.6}
The special case $\binom{N,0,\dots,0;0}{k_1,k_2,\dots,k_d;k_\infty}$ 
is the same as the usual multinomial coefficient 
$\binom{N}{k_1,k_2,\dots,k_d,k_\infty}$ in combinatorics.
Note also that $\binom{\bt}{\bk}\ne 0$ implies $t_\infty\le k_\infty$,
as the last factor of $\binom{\bt}{\bk}$ could survive only when 
$(t_1+\cdots t_d-k_1-\cdots-k_{d-1})-k_d=k_\infty-t_\infty\ge 0$.
\end{Remark}

It turns out that the array binomial coefficients give
the expansion of $\SS^{(\bt)}$ as a linear sum of the monomials in $W$.
Recall that, for $\bt=(t_1,\dots,t_d,t_\infty)\in  \N_0^{(\infty)}$, 
$w_\bt$ denotes the word $X^{t_1}YX^{t_2}Y\cdots X^{t_d}YX^{t_\infty}\in W$.

\begin{Lemma}[Monomial expansion]
\label{dS-mono}
$$
\SS^{(\bk)}=\sum_{\bt \in\N_0^{(\infty)}}
\binom{\bt}{\bk} w_{\bt}.
$$
\end{Lemma}
\begin{proof}
Without loss of generality, it suffices to show 
$\la w_{\bt}, \SS^{(\bk)} \ra=\binom{\bt}{\bk}$
in the case $(N:=)\,|\bt|=|\bk|$ and $(d:=)\,\dep(\bt)=\dep(\bk)$.
The assertion is trivial when $d=0$, as then
$\bk=\bt=(;N)$, $\SS^{(\bk)}=X^N=w_\bt$ and $\binom{\bt}{\bk}=1$.
For $d>0$, we argue by induction on $d$.
Suppose $d=1$, $\bk=(k_1;k_\infty)$ and $\bt=(t_1;t_\infty)$. Then
$$
\SS^{(\bk)}=
(X^{k_1}Y)\sha X^{k_\infty}
=\sum_{i=0}^{k_\infty}(X^{k_1}\sha X^i)YX^{k_\infty-i}
=\sum_{i=0}^{k_\infty}\binom{k_1+i}{k_1}
X^{k_1+i}YX^{k_\infty-i}.
$$
Since $N=k_1+k_\infty=t_1+t_\infty$, we have
$\la  w_{\bt}, \SS^{(\bk)} \ra=\binom{k_1+k_\infty-t_\infty}{k_1}
=\binom{t_1}{k_1}$.
Suppose $d>1$ with $\bk=(k_1,\dots,k_d;k_\infty)$ and
$\bt=(t_1,\dots,t_d;t_\infty)$.
Write $\bk'=(k_1,\dots,k_{d-1};0)\in \N_0^{(\infty)}
$. Then
\begin{align*}
\SS^{(\bk)}
&=(((\SS^{(\bk')}Y)\sha X^{k_d})Y)\sha X^{k_\infty} 
\\
&=\sum_{i=0}^{k_\infty}
\SS^{(\bk')}
Y(X^{k_d}\sha X^i)YX^{k_\infty-i} 
\qquad (\text{associativity of $\sha$})
\\
&=\sum_{i=0}^{k_\infty}
\sum_{\bt'}\binom{\bt'}{\bk'}
\binom{k_d+i}{k_d}
w_{\bt'}
X^{k_d+i}YX^{k_\infty-i},
\end{align*}
where $\bt'=(t_1',\dots,t_{d-1}';t_d')\in \N_0^{(\infty)}$ runs over
those tuples 
with $t_1'+\cdots+t_d'=|\bk'|$ 
so that $\SS^{(\bk')}$ is expressed as
$\sum_{\bt'}\binom{\bt'}{\bk'}w_{\bt'}$
by the induction hypothesis on $\dep(\bk')=d-1$.
The coefficient of $w_\bt$ in $\SS^{(\bk)}$
can be found in the above summand where
$k_\infty-i=t_\infty$,
$t_d'+k_d+i=t_d$ and
$t_s'=t_s$ ($s=1,\dots,d-1$), hence
$$
\la w_\bt, \SS^{(\bk)} \ra=
\binom{t_1}{k_1}
\cdots
\binom{t_1+\cdots t_{d-1}-k_1-\cdots-k_{d-2}}{k_{d-1}}
\cdot
\binom{k_d+k_\infty-t_\infty}{k_d}.
$$
Since $N=|\bk|=|\bt|$, we have
$k_d+k_\infty-t_\infty
=t_1+\cdots t_{d}-k_1-\cdots-k_{d-1}$.
This establishes the formula  
$\la w_\bt, \SS^{(\bk)} \ra=\binom{\bt}{\bk}$.
\end{proof}

\begin{Remark}
It would be worth noting that 
Lemma \ref{dS-mono} 
can be derived from
counting $\la w_{\bt}, \SS^{(\bk)} \ra$ 
as the number of certain shuffling of letters in
$w_\bk=X^{k_1}Y\cdots X^{k_d}YX^{k_\infty}$
to produce 
$w_\bt=X^{t_1}Y\cdots X^{t_d}YX^{t_\infty}$.
Assume 
$|\bt|=|\bk|$ and $\dep(\bt)=\dep(\bk)$,
and consider letters $Y$ as partitions between 
groups of letters $X$'s in $w_\bk$ and in $w_\bt$. 
Then $\la w_{\bt}, \SS^{(\bk)} \ra$ is the number of ways
of moving some letters $X$ in $w_\bk$
to the left (beyond any number of $Y$'s)
to form the word $w_\bt$ 
without changing orders between $X$'s from the same group
in $w_\bk$.
We count this number 
by enumerating branches of possibilities
for choosing
places of $X$'s in $w_\bt$ for those moved from 
$w_\bk$ group by group. 
The first binomial factor $\binom{t_1}{k_1}$
of (\ref{ArrayBino}) is the number of ways to choose
$k_1$ places for $X$'s (coming from 
the first group in $w_\bk$) in the first group $X^{t_1}Y$ of $w_\bt$.
The second binomial factor $\binom{t_1+t_2-k_1}{k_2}$ of (\ref{ArrayBino}) 
represents the number of ways to choose 
$k_2$ places for $X$'s (coming from the second group $YX^{k_2}Y$ in 
$w_\bk$) in the first two groups $X^{t_1}YX^{t_2}Y$ of $w_\bt$ 
where already occupied $k_1$ places in the previous step
are prohibited to choose.
We continue the process in the same way. For each given 
$i\in\{2,\dots,d\}$, suppose that destinations of $X$'s in 
$X^{t_1}Y\cdots X^{t_{i-1}}$ 
from $X^{k_1}Y\cdots YX^{k_{i-1}}Y$ has already been
chosen. Then, the $i$-th binomial
factor $\binom{t_1+\cdots t_i-k_1-\cdots-k_{i-1}}{k_i}$ 
of (\ref{ArrayBino})
represents the number of ways to choose $k_i$ places for 
$X$'s (coming from the $i$-th group $YX^{k_i}Y$ in $w_\bk$)
in 
$X^{t_1}Y\cdots YX^{t_i}Y$ (the first $i$ groups of $w_\bt$):
There are $t_1+\cdots+t_i$ places for $X$ in $X^{t_1}Y\cdots YX^{t_i}Y$
but already $k_1+\cdots+k_{i-1}$ places are occupied 
by earlier choices.
Performing the process till $i=d$ 
verifies the desired identity
$\la w_\bt, \SS^{(\bk)} \ra=\binom{\bt}{\bk}$.
\end{Remark}

\begin{proof}[Proof of Theorem \ref{orthogonality}]
It is not difficult to see from the formula
$Y^{(k)}=\sum_{i=0}^k (-1)^i\binom{k}{i}X^{k-i}YX^i$ (\cite[(4)]{M37})
that the expansion of the Magnus polynomial in monomials
is given by 
\begin{equation} \label{Mag-Mono}
\MM^{(\bk)}=\sum_{\bt \in\N_0^{(\infty)}}
\magnus{\bk}{\bt} w_{\bt} 
\end{equation}
with
\begin{equation}
\label{MagnusBino}
\magnus{\bk}{\bt}:=(-1)^{\sum_{i=1}^d (d-i+1)(k_i-t_i)}
\binom{k_1}{k_1-t_1} 
\binom{k_2}{k_1+k_2-t_1-t_2} \cdots
\binom{k_d}{\sum_{i=1}^d(k_i-t_i)}
\end{equation}
for $\bt:=(t_1,\dots, t_d;t_\infty)$, $\bk:=(k_1,\dots,k_d;k_\infty)$.
Since $\lp\SS^{(\bt)},\MM^{(\bk)}\rp=\sum_{\bu\in\N_0^{(\infty)}}
\lp\SS^{(\bt)},w_\bu\rp \lp\MM^{(\bk)},w_\bu\rp$, it suffices to show
\begin{equation} \label{target}
\sum_\bu \magnus{\bk}{\bu} \binom{\bu}{\bt}=\delta_\bk^\bt.
\end{equation}
Noting that non-zero pairing $\lp\SS^{(\bt)},\MM^{(\bk)}\rp$ occurs only when
$|\bt|=|\bk|$, $\dep(\bt)=\dep(\bk)$, without loss of generality, we may
assume that $\bu$ in the above summation also runs over those
with the fixed size $N:=|\bt|=|\bk|$ and depth $d:=\dep(\bt)=\dep(\bk)$.
Then, the summation $\sum_\bu$ with $\bu=(u_1,\dots,u_d;u_\infty)$ 
has $d$ independent parameters $u_1,\dots, u_d$ that determine 
$u_\infty=N-\sum_{i=1}^d u_i$. We may also regard each $u_i$ running over $\Z$,
as the coefficients $\magnus{\bk}{\bu}$, $\binom{\bu}{\bt}$ vanish when 
combinatorial meaning is lost.
Then, in the summation $\sum_{(u_1,\dots,u_d)\in\Z^d}$ in (\ref{target}),
the partial factor of summation 
involved with the last parameter $u_d$ can be factored out in the form:
\begin{align*}
&
\sum_{u_d\in\Z}(-1)^{-u_d}
\binom{k_d}{u_d+\sum_{i=1}^{d-1}(u_i-k_i)}
\binom{u_d+\sum_{i=1}^{d-1}(u_i-t_i)}{t_d} \\
&=
(-1)^{\sum_{i=1}^{d-1}(u_i-k_i)-k_d}
\binom{\sum_{i=1}^{d-1}(k_i-t_i)}{t_d-k_d}.
\end{align*}
(Use \cite[(5.24)]{GKP}.)
Repeating this process inductively on $d$, we eventually find
$$
\lp\SS^{(\bt)},\MM^{(\bk)}\rp=
\binom{0}{t_1-k_1}
\binom{k_1-t_1}{t_2-k_2}
\binom{k_1+k_2-t_1-t_2}{t_3-k_3}
\cdots
\binom{\sum_{i=1}^{d-1}(k_i-t_i)}{t_d-k_d}
$$
which is equal to $\delta^\bk_\bt$ as desired.
\end{proof}

\begin{Corollary} \label{coro}
Each element $u\in R\la X,Y\ra$ can be written as 
$$
u=\sum_{\bk \in\N_0^{(\infty)}}
\lp\SS^{(\bk)}, u\rp \,\MM^{(\bk)}
=\sum_{\bk \in\N_0^{(\infty)}}\lp\MM^{(\bk)}, u\rp \,\SS^{(\bk)}.
$$
\end{Corollary}

Note that only a finite number of summands are nonzero
in either summation above. 

\section{Generalization to the case $R\la X,Y_1,Y_2,\cdots \ra$}

It is not difficult to generalize the above duality in $R\la X,Y\ra$ 
(Theorem  \ref{orthogonality})
to similar duality in $R\la X,Y_\lambda \ra_{\lambda\in \Lambda}$ 
($\Lambda$: a nonempty index set), viz. in 
the associative algebra freely generated by
the symbols $X,Y_\lambda$ $(\lambda\in\Lambda)$ over $R$. 
In fact, introducing 
\begin{equation}
Y_\lambda^{(0)}:=Y_\lambda, \quad
Y_\lambda^{(k+1)}:=[X,Y_\lambda^{(k)}] \quad
(\lambda\in\Lambda, k=0,1,2,\dots)
\end{equation}
that are called the elements arising by elimination of $X$,
Magnus (\cite[Hilfssatz 2]{M37}, \cite[Lemma 5.6]{MKS})
showed that every element $Z$ of $R\la X,Y_\lambda \ra_{\lambda\in \Lambda}$ 
has the unique expression (\ref{1st-elimination}) 
with 
$S_X$ 
the subalgebra 
freely generated by the
$Y_\lambda^{(k)}$ ($k\in\N_0,\lambda\in\Lambda$).

\begin{Definition}[Depth-varied Magnus/demi-shuffle polynomials and monomials] 
\label{def3.1}
Let $d$ be a positive integer. 
For  $\bk=(k_1,\dots,k_d; k_\infty)\in\N_0^{(\infty)}$ and a 
finite sequence $\blambda=(\lambda_1,\dots,\lambda_d)\in\Lambda^d$, 
define
\begin{align*}
&\MM^{(\bk,\blambda)}:=Y_{\lambda_1}^{(k_1)}\cdots Y_{\lambda_d}^{(k_d)} \cdot X^{k_\infty};\\
&\SS^{(\bk,\blambda)}:=(\cdots((X^{k_1}Y_{\lambda_1})\sha X^{k_2})Y_{\lambda_2})\sha\cdots\sha ) X^{k_d})Y_{\lambda_d})\sha X^{k_\infty}; \\
&w_{\bk,\blambda}:=X^{k_1}Y_{\lambda_1}\cdots X^{k_d}Y_{\lambda_d}X^{k_\infty}.
\end{align*}
For $d=0$ with $\bk=(;k)$, $\blambda=()$, 
we simply set $w_{(;k),()}=\MM^{((;k),())}=\SS^{((;k),())}=X^k$.
\end{Definition}

Note that the monomials $w_{\bk,\blambda}$ ($\bk\in\N_0^{(\infty)}$,
$\blambda\in\Lambda^{\dep(\bk)})$
form an $R$-linear basis of $R\la X,Y_\lambda \ra_{\lambda\in \Lambda}$.
Let us write $\la\ ,\ \ra$ for the standard pairing defined by the Kronecker
symbol with respect to these monomials.

\begin{Theorem}[Duality] \label{orthogonality2}
For $\bt, \bk\in \N_0^{(\infty)}$ and $\blambda\in\Lambda^{\dep(\bt)}$,
$\bmu\in\Lambda^{\dep(\bk)}$, we have
$$
\lp \SS^{(\bt,\blambda)},\MM^{(\bk,\bmu)}\rp=\delta_{(\bt,\blambda)}^{(\bk,\bmu)}.
$$
Here $\delta_{(\bt,\blambda)}^{(\bk,\bmu)}$ is the Kronecker symbol, 
i.e., designating $1$ or $0$
according to whether the pairs $(\bt,\blambda)$ and $(\bk,\bmu)$ 
coincide or not respectively.
\end{Theorem}

\begin{proof}
Given a fixed $\blambda=(\lambda_1,\dots,\lambda_d)\in\Lambda^d$, 
let $V_\blambda$ be the $R$-linear 
subspace of $R\la X,Y_\lambda \ra_{\lambda\in \Lambda}$ generated by
the monomials $\{w_{\bk,\blambda}\mid\bk\in\N_0^{(\infty)}, \dep(\bk)=d\}$.
It is obvious that if $\blambda\ne\bmu$ then 
$V_\blambda$ and $V_\bmu$ are mutually orthogonal under the standard pairing $\la\ ,\ \ra$.
Since $\MM^{(\bk,\bmu)}\in V_\bmu$, $\SS^{(\bt,\blambda)}\in V_\blambda$,
we only need to look at the case $\bmu=\blambda\in \Lambda^d$. 
Consider the $R$-linear subspace $V_d$ of $R\la X,Y\ra$ generated by 
$\{w_\bk\mid \bk\in \N_0^{(\infty)}, \dep(\bk)=d\}$. Then, the mapping
$w_\bk\mapsto w_{\bk,\blambda}$ defines an isometry, i.e., an $R$-linear 
isomorphism $\phi_\blambda:V_d \isom V_\blambda$ preserving $\la\ ,\ \ra$.
The assertion then follows at once from Theorem \ref{orthogonality}
after observing $\phi_\blambda( \SS^{(\bt)})=\SS^{(\bt,\blambda)}$ and
$\phi_\blambda(\MM^{(\bk)})=\MM^{(\bk,\blambda)}$.
\end{proof}

\section{Application to a formula of Le-Murakami and Furusho type}

In this section, we assume that $R$ is a field and consider 
$R\la X,Y\ra$ as a subalgebra of the ring of non-commutative
formal power series $R\lala X,Y\rara$, where a standard
comultiplication $\Delta$ is defined by setting
$\Delta(a)=1\otimes a+a\otimes 1$ for $a\in\{X,Y\}$.  
An element $J\in R\lala X,Y\rara$ is called group-like if
it has constant term 1 and satisfies $\Delta(J)=J\otimes J$. 
There are many group-like elements; 
for example, the subgroup multiplicatively generated by $\exp(X)$ and $\exp(Y)$
in $R\lala X,Y\rara^\times$ consists of group-like elements and forms a free group of rank 2. 

\begin{Theorem}[Le-Murakami, Furusho type formula]
\label{LeMuFu}
Let $J\in R\lala X,Y\rara$ be a group-like element in the form
$$
J=\sum_{\bk\in\N_0^{(\infty)}} c_\bk w_\bk,
$$
and write $c_X$ for the coefficient $c_{(;1)}$ of $X$ in $J$. 
Then,
$$
c_{(k_1,\dots,k_d;k_\infty)}=\sum_{\substack{s,t\ge 0 \\ s+t=k_\infty}}
(-1)^s
\frac{{(c_X)}^t}{t!}
\sum_{\substack{ s_1,\dots,s_d\ge 0 \\ s=s_1+\dots+s_d}}
\binom{k_1+s_1}{k_1}\cdots \binom{k_d+s_d}{k_d}
c_{(k_1+s_1,\dots,k_d+s_d;0)}
\ .
$$
\end{Theorem}

We first prove an elementary identity that will be used for the proof of the above formula.

\begin{Lemma} \label{lem2}
Let $\bkappa=(k_1,\dots,k_d)\in\N_0^d$ and $\bs=(s_1,\dots,s_d)\in \Z^d$
satisfy $s=s_1+\cdots+s_d\ge 0$ and $k_i+s_i\ge 0$ $(i=1,\dots,d)$.
Then,  we have
$$
\sum_{\btau\in\N_0^d} 
\lp\SS^{(\btau;0)},w_{(\bkappa+\bs;0)}\rp
\cdot \lp\MM^{(\btau;0)}, w_{(\bkappa;s)}\rp
=(-1)^s
\binom{k_1+s_1}{k_1}\cdots \binom{k_d+s_d}{k_d}.
$$
\end{Lemma}

\begin{proof}
We shall compute the LHS explicitly
as the sum over 
$\btau\in\N_0^d$ satisfying
$\sum_{i=1}^d t_i=\sum_{i=1}^d (k_i+s_i)$
with
$$
\la\SS^{(\btau;0)},w_{(\bkappa+\bs;0)}\ra=\binom{(\bkappa+\bs;0)}{(\btau;0)}
=\binom{k_1+s_1}{t_1} 
\cdots
\binom{\sum_{i=1}^{d-1}(k_i+s_i)-\sum_{i=1}^{d-2}t_i}{t_{d-1}}
\binom{t_d}{t_d}
$$
by Lemma \ref{dS-mono} and with
$$
\la\MM^{(\btau;0)}, w_{(\bkappa;s)}\ra
=\magnus{(\btau;0)}{(\bkappa;s)}
=(-1)^{s+\sum_{i=1}^{d-1}(d-i)(t_i-k_i)}
\binom{t_1}{t_1-k_1}
\cdots
\binom{t_{d-1}}{\sum_{i=1}^{d-1}(t_i-k_i)}
\binom{t_d}{s}
$$
by (\ref{Mag-Mono}) and
$s=\sum_{i=1}^{d}(t_i-k_i)$. 
Note that,
since $\binom{(\bkappa+\bs;0)}{(\btau;0)}\magnus{(\btau;0)}{(\bkappa;s)}\ne 0$ 
only when all entries of $\btau=(t_1,\dots,t_d)$ are nonnegative
and $t_1+\cdots+t_d=\sum_{i=1}^d (k_i+s_i)$ (constant),
the above sum
can be taken over the tuples $(t_1,\dots,t_{d-1})\in\Z^{d-1}$
with entries running as independent integers.
Then, the partial summation involved with the last variable 
$t_{d-1}$ may be factored out as
\begin{align*}
&\sum_{t_{d-1}}(-1)^{t_{d-1}}
\binom{\sum_{i=1}^{d-1}(k_i+s_i)-\sum_{i=1}^{d-2}t_i}{t_{d-1}}
\binom{t_{d-1}}{\sum_{i=1}^{d-1}(t_i-k_i)}
\binom{t_d}{s} \\
&=\sum_{t_{d-1}}(-1)^{t_{d-1}}
\binom{\sum_{i=1}^{d-1}(k_i+s_i)-\sum_{i=1}^{d-2}t_i}{ \sum_{i=1}^{d-1}k_i-\sum_{i=1}^{d-2}t_i }
\binom{\sum_{i=1}^{d-1} s_i}{\sum_{i=1}^{d-1}(t_i-k_i)}
\binom{\sum_{i=1}^d(k_i+s_i)-\sum_{i=1}^{d-1}t_i}{\sum_{i=1}^{d} s_i} \\
&=\binom{\sum_{i=1}^{d-1}(k_i+s_i)-\sum_{i=1}^{d-2}t_i}{\sum_{i=1}^{d-1} s_i}
(-1)^{\sum_{i=1}^{d-1}k_i -\sum_{i=1}^{d-2} t_i}
\binom{k_d+s_d}{s_d},
\end{align*}
where \cite[(5.21)]{GKP} is applied for the first equality and 
\cite[(5.24)]{GKP} for the second.
After factoring out the constant $\binom{k_d+s_d}{s_d}$ and repeating
the similar process with the other variables $t_{d-2},\dots, t_1$ consecutively,
we eventually obtain the asserted formula.
Below in Note \ref{Quest4.3}, we also provide an alternative proof of the lemma 
free from intricate use of 
\cite[(5.21),(5.24)]{GKP}.
\end{proof}

\begin{proof}[Proof of Theorem \ref{LeMuFu}]
We argue in the beautiful framework exploited in Reutenauer's book \cite[1.5]{Reu}
using the complete tensor product
$$
\mathscr{A}=R\lala X,Y\rara \bar\otimes R\lala X,Y\rara
$$
equipped with a product induced from 
the shuffle product (resp. the concatenation product)
on the left (resp. right) of $\bar\otimes$. 
Recall that the ring of
$R$-linear endomorphism $\mathrm{End}_RR\lala X,Y\rara$
can be embedded into $\mathscr{A}$ by 
$f\mapsto \sum_{w\in W} w\otimes f(w)$, and that the product of 
$\mathscr{A}$ restricts to the convolution product of $\mathrm{End}_RR\lala X,Y\rara$
defined by $f\ast g:=\mathrm{conc}\circ(f\otimes g)\circ \Delta$
(`$\mathrm{conc}$' means concatenation of left and right sides of $\otimes$).
Note that, for $f\in\mathrm{End}_R R\lala X,Y\rara$ and  $J\in R\lala X,Y\rara$,
we have $f(J)=\sum_{w\in W} \lp w,J\rp f(w)$.

Since, by Corollary \ref{coro}, every word $w$ can be written as 
$\sum_{\bt\in \N_0^{(\infty)}}
\lp \SS^{(\bt)},w \rp \MM^{(\bt)}$, 
the element of $\mathscr{A}$ corresponding to 
the identity $\mathrm{id}\in \mathrm{End}_RR\lala X,Y\rara$ is:
\begin{align*}
\sum_{w\in W} w\otimes w &=
\sum_w w\otimes \sum_\bt \lp \SS^{(\bt)},w \rp\MM^{(\bt)} 
=
\sum_\bt (\sum_w  \lp\SS^{(\bt)},w\rp w) \otimes \MM^{(\bt)} \\
&= 
\sum_\bt \SS^{(\bt)} \otimes \MM^{(\bt)} \\
&=
\left(\sum_{d=0}^\infty 
\sum_{\btau\in\N_0^d} \SS^{(\btau;0)}\otimes \MM^{(\btau;0)}\right)
\cdot
\left(\sum_{t=0}^\infty X^t\otimes X^t\right),
\end{align*}
where used are 
$\SS^{(\bt)}=\SS^{(\btau;t)}=\SS^{(\btau;0)}\sha X^t$ and 
$\MM^{(\bt)}=\MM^{(\btau;t)}=\MM^{(\btau;0)}\cdot X^t$.
Observing that both factors of the above last side correspond to specific 
$R$-linear endomorphisms,
we can apply $\mathrm{id}$ to $J$ as the convolution product of them and find from
$\Delta(J)=J\otimes J$ that
\begin{equation} \label{J-expansion}
J=\mathrm{id}(J)=
\left(\sum_{d=0}^\infty 
\sum_{\btau\in\N_0^d} \lp \SS^{(\btau;0)},J \rp \MM^{(\btau;0)}\right)
\left(\sum_{t=0}^\infty \frac{(c_X)^t}{t !} X^t\right).
\end{equation}
Note here that the pairing of $J$ with 
$X^t=X^{\sha\, t}/t!$ is equal to $(c_X)^t/t !$,
as easily seen from the fact that the specialization 
$J(X,0)\in R\lala X\rara$ at $Y=0$ is a group like
element $\exp(c_X\cdot X)$.
To settle the proof of Theorem \ref{LeMuFu},
given a fixed $\bk=(\bkappa;k_\infty)=(k_1,\dots,k_d; k_\infty)\in\N_0^{(\infty)}$
and $0\le s\le k_\infty$, 
we compute
the coefficient of $w_{(\kappa;s)}=X^{k_1}Y\cdots X^{k_d}YX^s$ in the expansion 
of the first factor of the above right hand side as follows:
\begin{align*}
&\sum_{d=0}^\infty 
\sum_{\btau\in\N_0^d} 
\lp \SS^{(\btau;0)},J \rp 
\lp \MM^{(\btau;0)},w_{(\bkappa;s)} \rp 
=\left\lp\sum_{d=0}^\infty 
\sum_{\btau\in\N_0^d} \lp\SS^{(\btau;0)},J \rp \MM^{(\btau;0)},
w_{(\bkappa;s)} \right\rp \\
=&\left\lp\sum_{d=0}^\infty 
\sum_{\btau\in\N_0^d} 
\biggl\lp\SS^{(\btau;0)},\sum_{\bu\in \N_0^{(\infty)}}(J,w_\bu)w_\bu\biggr\rp
 \MM^{(\btau;0)},
w_{(\bkappa;s)}\right\rp \\
=&\sum_{\bu} \lp J,w_\bu \rp \sum_{d=0}^\infty 
\sum_{\btau\in\N_0^d} \lp \SS^{(\btau;0)},w_\bu \rp
\lp \MM^{(\btau;0)}, w_{(\bkappa;s)} \rp .
\end{align*}
But since $\lp\SS^{(\btau;0)},w_\bu\rp \lp \MM^{(\btau;0)}, w_{(\bkappa;s)} \rp$ 
survives only when 
$\dep(\btau;0)=\dep(\bkappa;s)=\dep(\bu)$ and $|(\btau;0)|=|(\bkappa;s)|=|\bu|$,
the summation $\sum_\bu$ in the above last side occurs only for 
those $\bu$ of the form 
$(\bkappa+\bs;0)\in \N_0^{(\infty)}$ 
with $\bs=(s_1,\dots,s_d)\in\Z^d$, 
$s=s_1+\cdots+s_d\ge 0$
(cf. also Remark \ref{rem2.6}).  
Then, it follows from Lemma \ref{lem2} that the above last side is equal to 
$$
\sum_{d=0}^\infty \sum_{\substack{\bs\in\N_0^d \\ |(\bs;0)|=s}} 
\lp J, w_{(\bkappa+\bs;0)} \rp
(-1)^s
\binom{k_1+s_1}{k_1}\cdots \binom{k_d+s_d}{k_d}.
$$
(Note: The prescribed condition $\bs\in\Z^d$ has been replaced 
with $\bs\in\N_0^d$ for a posteriori survivals of binomial factors).
{}From this and (\ref{J-expansion}) together with
$\lp J, w_{(\bkappa+\bs;0)} \rp=c_{(k_1+s_1,\dots,k_d+s_d;0)}$, 
we conclude the assertion.
\end{proof}

\begin{Note}[{\it Alternative proof of Lemma} \ref{lem2}] \label{Quest4.3}
In the right hand side of Lemma \ref{lem2}, the quantity
$\binom{k_1+s_1}{k_1}\cdots \binom{k_d+s_d}{k_d}$ can be interpreted as
the pairing
$\lp w_{(k_1,\dots,k_d;0)}\sha X^s, w_{(k_1+s_1,\dots,k_d+s_d;0)} \rp$.
Therefore, the assertion of Lemma is equivalent to the identity
\begin{equation}
\sum_{\btau\in\N_0^d} 
\lp\SS^{(\btau;0)},w_{(\bkappa+\bs;0)} \rp \cdot \lp\MM^{(\btau;0)}, w_{(\bkappa;0)}\!\cdot\!\! X^s\rp
=(-1)^s 
\lp w_{(\bkappa;0)}\sha X^s, w_{(\bkappa+\bs;0)} \rp
\label{alterLemma}
\end{equation}
for $\bkappa=(k_1,\dots,k_d)\in\N_0^d$, $\bs=(s_1,\dots,s_d)\in\Z^d$ 
satisfying $s=s_1+\cdots+s_d\ge 0$ and $\bkappa+\bs\in\N_0^d$.
We now give an alternative proof for it using the Magnus/demi-shuffle duality:
First, by Corollary \ref{coro}, we have 
$w_{(\bkappa;0)}=\sum_\bbr\la \MM^{(\bbr)},w_{(\bkappa;0)}\ra \SS^{(\bbr)}$ 
and $w_{(\bkappa+\bs;0)}=\sum_\bt\la \SS^{(\bt)},w_{(\bkappa+\bs;0)}\ra \MM^{(\bt)}$
so that the RHS of (\ref{alterLemma}) can be written as
\begin{align}
\label{RHS4.2}
&(-1)^s 
\lp w_{(\bkappa;0)}\sha X^s, w_{(\bkappa+\bs;0)} \rp \\
&=(-1)^s \sum_{\bbr,\bt\in\N_0^{(\infty)}}
\lp \SS^{(\bbr)}\sha X^s, \MM^{(\bt)} \rp
\lp \MM^{(\bbr)} ,w_{(\bkappa;0)} \rp
\lp \SS^{(\bt)} ,w_{(\bkappa+\bs;0)}\rp
\notag \\
&=(-1)^s \sum_{\brho\in\N_0^d} 
\lp \MM^{(\brho;0)} ,w_{(\bkappa;0)}\rp
\lp  \SS^{(\brho;s)}, w_{(\bkappa+\bs;0)} \rp.
\notag
\end{align}
Here in the second equality, we use the fact that
$\lp  \MM^{(\bbr)}, w_{(\bkappa;0)} \rp$ 
survives only if $\bbr=(\brho;0)\in\N_0^{(\infty)}$ for some $\brho\in\N_0^d$
and then apply the duality (Theorem \ref{orthogonality})
to $\lp \SS^{(\bbr)}\sha X^s, \MM^{(\bt)} \rp$
with $\SS^{(\brho;0)}\sha X^s=\SS^{(\brho;s)}$ 
(cf. Definitions \ref{defMagnus} and \ref{demi-shuffle}). 

On the other hand, in the LHS of  (\ref{alterLemma}),
one observes that nontrivial terms of the summation arise
only from those $\btau=(\tau_1,\dots,\tau_d)\in\N_0^d$ subject to
$\sum_{i=1}^d \tau_i=s+\sum_{i=1}^d \kappa_i$ (constant). 
But then, the last binomial factor in (\ref{MagnusBino}) for
$ \lp\MM^{(\btau;0)}, w_{(\bkappa;0)}\!\cdot\!\! X^s\rp
=\magnus{(\tau_1,\dots,\tau_d;0)}{(\kappa_1,\dots,\kappa_d;s)}
$
equals $\binom{\tau_d}{s}$ which is non-zero only if
$\tau_d\ge s$.
Therefore, 
the summation $\sum_\btau$ may be replaced by 
$\sum_\brho$ with $\brho=\btau-(\mathbf{0},s)$
in $\N_0^d$ (where $\mathbf{0}\in\N_0^{d-1}$: the zero vector).
Thus,
the LHS of (\ref{alterLemma}) can be written as
\begin{align}
\label{LHS4.2}
&\sum_{\btau\in\N_0^d} 
\lp\SS^{(\btau;0)},w_{(\bkappa+\bs;0)} \rp \cdot 
\lp\MM^{(\btau;0)}, w_{(\bkappa;0)}\!\cdot\!\! X^s\rp \\
&=
\sum_{\brho\in\N_0^d} 
\lp\SS^{(\brho+(\mathbf{0},s);0)},w_{(\bkappa+\bs;0)} \rp \cdot 
\lp\MM^{(\brho+(\mathbf{0},s);0)}, w_{(\bkappa;s)} \rp. 
\notag
\end{align}
Comparing summands of the above (\ref{RHS4.2}) and (\ref{LHS4.2}) for 
individual $\brho\in\N_0^d$ in view of coefficients of 
monomial expansions of demi-shuffle/Magnus polynomials 
(Lemma \ref{dS-mono} and (\ref{Mag-Mono})), we reduce 
the formula (\ref{alterLemma}) to the following elementary identity
for $\bkappa=(k_i),\brho=(r_i)\in\N_0^d$ and $\bs=(s_i)\in\Z^d$
satisfying
$\sum_{i=1}^d k_i=\sum_{i=1}^d r_i$, $\bs+\bkappa\in \N_0^d$
and $s:=\sum_{i=1}^d s_i\ge 0$:
\begin{equation}
\binom{(\bkappa+\bs;0)}{(\brho;s)}
\magnus{(\brho;0)}{(\bkappa;0)}
=
(-1)^s
\binom{(\bkappa+\bs;0)}{(\brho+(\mathbf{0},s);0)}
\magnus{(\brho+(\mathbf{0},s);0)}{(\bkappa,s)}
\end{equation} 
that is an immediate consequence of definitions
of these symbols $\{^*_* \}$, $(^*_*)$.
(Observe that only difference between the corresponding symbols
occurs from the last binomial coefficient in  (\ref{ArrayBino})
and (\ref{MagnusBino}).)
\qed
\end{Note}

\begin{Example}
The following shows an output of a group-like element 
$J=\sum_{w\in W}c_w w$ of $R\lala X,Y\rara$ with the
shuffle relation  (which is necessary and 
sufficient for group-likeness due to Ree \cite{Ree}) 
counted from a computation using software \cite{Maple}
up to total degree 4.

$ \displaystyle J=
1+c_{X} X +c_{Y} Y +\frac{c_{X}^{2} \mathit{XX}}{2}+c_{\mathit{XY}} \mathit{XY} +\left(c_{X} c_{Y}-c_{\mathit{XY}}\right) \mathit{YX} 
+\frac{c_{Y}^{2} \mathit{YY}}{2}+\frac{c_{X}^{3} \mathit{XXX}}{6} \\
+c_{\mathit{XXY}} \mathit{XXY} +\left(c_{X} c_{\mathit{XY}}-2 c_{\mathit{XXY}}\right) \mathit{XYX} +c_{\mathit{XYY}} \mathit{XYY} +\left(\frac{1}{2} c_{X}^{2} c_{Y}-c_{X} c_{\mathit{XY}}+c_{\mathit{XXY}}\right) \mathit{YXX} 
\\
+\left(c_{\mathit{XY}} c_{Y}-2 c_{\mathit{XYY}}\right) \mathit{YXY} 
+\left(\frac{1}{2} c_{X} c_{Y}^{2}-c_{\mathit{XY}} c_{Y}+c_{\mathit{XYY}}\right) \mathit{YYX} +\frac{c_{Y}^{3} \mathit{YYY}}{6}
\\
+\frac{c_{X}^{4} \mathit{XXXX}}{24}+c_{\mathit{XXXY}} \mathit{XXXY} +\left(c_{X} c_{\mathit{XXY}}-3 c_{\mathit{XXXY}}\right) \mathit{XXYX} +c_{\mathit{XXYY}} \mathit{XXYY} 
\\
+\left(\frac{1}{2} c_{X}^{2} c_{\mathit{XY}}-2 c_{X} c_{\mathit{XXY}}+3 c_{\mathit{XXXY}}\right) \mathit{XYXX} +\left(\frac{c_{\mathit{XY}}^{2}}{2}-2 c_{\mathit{XXYY}}\right) \mathit{XYXY} 
\\
+\left(c_{X} c_{\mathit{XYY}}-\frac{c_{\mathit{XY}}^{2}}{2}\right) \mathit{XYYX} 
+c_{\mathit{XYYY}} \mathit{XYYY} +\left(\frac{1}{6} c_{X}^{3} c_{Y}-\frac{1}{2} c_{X}^{2} c_{\mathit{XY}}+c_{X} c_{\mathit{XXY}}-c_{\mathit{XXXY}}\right) \mathit{YXXX}
\\
+\left(c_{\mathit{XXY}} c_{Y}-\frac{c_{\mathit{XY}}^{2}}{2}\right) \mathit{YXXY} 
+\left(c_{X} c_{\mathit{XY}} c_{Y}-2 c_{X} c_{\mathit{XYY}}-2 c_{\mathit{XXY}} c_{Y}+\frac{1}{2} c_{\mathit{XY}}^{2}+2 c_{\mathit{XXYY}}\right) \mathit{YXYX}
\\ 
+\left(c_{\mathit{XYY}} c_{Y}-3 c_{\mathit{XYYY}}\right) \mathit{YXYY} 
+\left(\frac{1}{4} c_{X}^{2} c_{Y}^{2}-c_{X} c_{\mathit{XY}} c_{Y}+c_{X} c_{\mathit{XYY}}+c_{\mathit{XXY}} c_{Y}-c_{\mathit{XXYY}}\right) \mathit{YYXX} 
\\
+\left(\frac{1}{2} c_{\mathit{XY}} c_{Y}^{2}-2 c_{\mathit{XYY}} c_{Y}+3 c_{\mathit{XYYY}}\right) \mathit{YYXY} 
+\left(\frac{1}{6} c_{X} c_{Y}^{3}-\frac{1}{2} c_{\mathit{XY}} c_{Y}^{2}+c_{\mathit{XYY}} c_{Y}-c_{\mathit{XYYY}}\right) \mathit{YYYX}
\\
+\frac{c_{Y}^{4} \mathit{YYYY}}{24} 
\quad+\quad $(terms of degree $\ge 5$).

\medskip
In the above computation, one observes that the coefficient $c_{XYXY}$
is expressed by lower simpler coefficients of $J$. This does not follow 
from Theorem \ref{LeMuFu}, however, does reflect the fact 
that $XYXY$ is not a Lyndon word.
Discussions on the most economical expression using only the 
coefficients of Lyndon words can be found in \cite{MPH}. 
\end{Example}

\begin{Note}
In the modern theory of multiple zeta values, a certain standard solution 
$G_0^z(X,Y)\in \C\lala X,Y\rara$ to the KZ-equation on $z\in \C-\{0,1\}$
is known as the generating function for the multiple polylogarithms (MPL).
It is also used to define the Drinfeld associator $\Phi(X,Y) \in\C\lala X,Y\rara$.  
The coefficients of $w_{(k_1,\dots,k_d;0)}$ in $\Phi(X,Y)$ (resp. in $G_0^z(X,Y)$)
are regular multiple zeta values (resp. regular MPL) of multi-index $(k_1,\dots,k_d)$,
but the other coefficients are in general not.  
Le-Murakami \cite{LM}, Furusho \cite{F} derived formulas that express
all coefficients of $\Phi(X,Y)$ and $G_0^z(X,Y)$ by those `regular' coefficients
explicitly. 
In \cite[Remark 2]{N21}, the author posed a question if it could be a similar
case for `$\ell$-adic Galois associator $f_\sigma^z(X,Y)\in\Q_\ell\lala X,Y\rara$',
in which context analytic theory of KZ-equation is unavailable yet. 
Since $f_\sigma^z(X,Y)$ is by definition a group-like element, 
the above Theorem \ref{LeMuFu} answers the question affirmatively.
\end{Note}

\begin{Note}
\label{zinbiel}
A noteworthy notion closely related to our $\SS^{(\bk)}$,
$\SS^{(\bk,\blambda)}$ is
the free Zinbiel (or, dual Leibniz) algebra 
studied by J.-L. Loday \cite{Lo95}, I. Dokas \cite{Do10}, F. Chapoton \cite{Ch21} et.al.
Let $V$ be a vector space with a basis $\mathfrak{B}=\{X_0,X_1,\dots\}$ and 
$T(V)$ be the tensor algebra (free associative algebra) generated by the
letters in $\mathfrak{B}$.  Loday introduced the ``half-shuffle'' product $\prec$
in $T(V)$ as the linear extension of the binary product on words 
given by:
$$
(x_0x_1\cdots x_p)\prec (x_{p+1}\cdots x_{p+q})
:=
x_0\cdot \bigl( (x_1\cdots x_p)\sha (x_{p+1}\cdots x_{p+q})\bigr),
$$
where $x_i$ are letters in $\mathfrak{B}$ ($i=0,\dots, p+q$).
It is remarkable that, while the usual shuffle product $w\,\sha\, w'=w\prec w'+w'\prec w$ 
is associative (and commutative), the half-shuffle product 
$\prec$ is not even associative --- however satisfying 
$(w_1\prec w_2)\prec w_3=w_1\prec (w_2\prec w_3)+w_1\prec (w_3\prec w_2)$.
We may relate the `Zinbiel monomials' 
with our demi-shuffle polynomials 
$\SS^{(\bk,\blambda)}$ in Definition \ref{def3.1} as follows:
Write $\ast\mapsto \overline{\ast}$ for the anti-automorphism of 
$R\la X,Y_\lambda\ra_{\lambda\in\Lambda}$
reversing the order of letters in each word, e.g., 
$\overline{XXY_\lambda}=Y_\lambda XX$.
Then,  
\begin{equation}
\label{semi-shuffle}
\overline{\SS^{(\bk,\blambda)}}=
X^{k_\infty}\sha \bigl(...
\bigl(Y_{\lambda_d}X^{k_d} \prec\bigl( Y_{\lambda_{d-1}}X^{k_{d-1}} \prec\bigl(
\cdots 
\prec \bigl( Y_{\lambda_2}X^{k_2} \prec Y_{\lambda_1}X^{k_1} \bigr)\bigr)...\bigr)
\end{equation}
for $\bk=(k_1,\dots,k_d; k_\infty)\in\N_0^{(\infty)}$,
$\blambda=(\lambda_1,\dots,\lambda_d)\in\Lambda^d$.
These polynomials also appeared in \cite[Proposition 5.10]{Minh2019} to 
illustrate the coefficients (of the main factor) of a solution of the 
KZ-equation expanded in 
$(\mathrm{ad}_{-X}^{k_1}Y)\cdots(\mathrm{ad}_{-X}^{k_d}Y)$.
We also learn from a paper by L. Foissy and F. Patras \cite{FP13} 
that already in M.-P. Sch\"utzenberger's work \cite{Sch}
is found an axiomatic treatment of half-shuffle combinatorics 
on words named ``alg\`ebre de shuffle''.

Calling  $\SS^{(\bk)}$, $\SS^{(\bk,\blambda)}$ 
`demi-shuffle' in Definitions \ref{demi-shuffle}, \ref{def3.1} 
or reserving `semi-shuffle' for 
names of anything else
might keep a moderate distance from the already overwhelming naming
`half-shuffle' of the operation $\prec$ in literature.
 \end{Note}

\medskip
{\it Acknowledgement}:
The author is grateful to Hidekazu Furusho for hinting a positive answer to the 
question posed in \cite[Remark 2]{N21} toward the form of
Theorem \ref{LeMuFu}
of the present paper, and for valuable comments and information 
on what is mentioned in part of Note \ref{zinbiel}.
He also thanks Densuke Shiraishi for stimulating discussions that
share awareness of various open problems around $\ell$-adic Galois multiple polylogarithms.  
The author would like to express his gratitude to the referees for useful 
comments that helped to improve the presentation of this paper.
This work was supported by JSPS KAKENHI Grant Numbers JP20H00115.

\ifx\undefined\bysame
\newcommand{\bysame}{\leavevmode\hbox to3em{\hrulefill}\,}
\fi

\end{document}